\documentclass[a4paper,12pt]{article} 
\usepackage{amsmath,fancyhdr,amssymb,amsthm,geometry,enumerate,graphicx,amsfonts,mathrsfs,color}
\usepackage{footnote}
\usepackage{ulem}
\usepackage[colorlinks=true]{hyperref}

\author{Mohammad Salman$^{a,*}$, Ruchi Das$^a$}
\title{Dynamics of multi-sensitive non-autonomous systems with respect to a vector}
\theoremstyle{definition}
\newtheorem{defn}{Definition}[section]
\providecommand{\keywords}[1]{\textbf{Keywords :} #1}
\providecommand{\msc}[1]{\textbf{Mathematics Subject Classification(2020)} #1}
\theoremstyle{plain}
\newtheorem{thm}{Theorem}[section]

\newtheorem{cor}{Corollary}[section]
\newtheorem{lem}{Lemma}[section]
\theoremstyle{definition}
\newtheorem{exm}{Example}[section]
\newtheorem{rmk}{Remark}[section]

\newcommand{\N}{\mathbb{N}}
\newcommand{\f}{f_{1,\infty}}
\newcommand{\M}{\mathcal{M}(X)}
\newcommand{\W}{\widetilde{f}_{1,\infty}}
\begin{document}
\date{}
\maketitle

\begin{abstract} We introduce the concept of multi-sensitivity with respect to a vector for a non-autonomous discrete system. We prove that for a periodic non-autonomous system on the closed unit interval, sensitivity is equivalent to strong multi-sensitivity and justify that the result need not be  true if the system is not periodic. In addition, we study strong multi-sensitivity and $\mathcal{N}$-sensitivity on non-autonomous systems induced by probability measure spaces.  Moreover,  we first prove that if $f_n$ converges to $f$ uniformly, then strong multi-sensitivity  (respectively, $\mathcal{N}$-sensitivity) of the non-autonomous system does not coincide with that of $(X, f)$. Then we give a sufficient condition such that non-autonomous system  $\{f_n\}_{n=1}^\infty = \f$  is strongly multi-sensitive (respectively, $\mathcal{N}$-sensitive) if and only if $f$ is so. Finally, we prove that if a non-autonomous system converges uniformly, then multi-transitivity and dense periodicity imply  $\mathcal{N}$-sensitivity.
\end{abstract}


\msc{37B20; 37B55; 54B10; 54H20}
\bigskip\renewcommand{\thefootnote}{\fnsymbol{footnote}}
\footnotetext{\hspace*{-5mm}
\renewcommand{\arraystretch}{1}
\begin{tabular}{@{}r@{}p{13cm}@{}}
$^*$& the corresponding author. \textit{Email addresses}: salman25july@gmail.com (M. Salman), rdasmsu@gmail.com (R. Das)\\
$^a$&Department of Mathematics, University of Delhi, Delhi-110007, India
\end{tabular}}
\vspace{-2mm}
\section{Introduction}
 Sensitive dependence on initial conditions (sensitivity), originated from the works of Ruelle and Takens \cite{RT1971}, describes the unpredictability in chaotic dynamical systems and is essence for different forms of chaos.   Auslander and Yorke \cite{AY1980}, later applied this idea of sensitivity to topological dynamical systems  and then Devaney popularized sensitivity by introducing Devaney chaos \cite{D1989}. Since then sensitivity has been a central topic when we talk about chaoticity and has attracted many researchers. In order to study unpredictability in a dynamical system on a larger scale, Moothathu \cite{M2007}, initiated the study of stronger forms of sensitivity in terms of subsets of $\mathbb{Z}^+$. He introduced the notions of cofinite sensitivity, multi-sensitivity, syndetic sensitivity and obtained many interesting results in terms of transitivity and minimality and proved that sensitivity implies cofinite sensitivity on the unit interval $[0,1]$.  Ever since the study of stronger forms of sensitivity, in particular multi-sensitivity has attracted the attention of many researchers. Wu et al. \cite{WWC2015} studied multi-sensitivity for product and hyperspatial dynamical systems. Later, in 2016, Huang et al. \cite{HKKZ2016} studied multi-sensitivity in terms of Lyapunov numbers and proved that if a dynamical system is multi-sensitive, then it has positive topological sequence entropy and proved that  for $M$-systems,  thickly syndetic sensitivity, thick sensitivity and multi-sensitivity are equivalent. Moothathu defined another interesting notion of multi-transitivity for dynamical systems \cite{M2010} which was later studied by Kwietniak and Oprocha and they proved that there is no relation between weakly mixing and multi-transitivity \cite{KO2012}. Later, in \cite{CLL2014}, Chen et al. introduced the concept of multi-transitivity with respect to a vector and gave a characterization for multi-transitivity. In 2016, Wu studied multi-transitivity on induced dynamical systems on the space of probability measures \cite{W2016}. Motivated by the works of Moothathu \cite{M2007, M2010} and Chen et al. \cite{CLL2014}, Jiao et al., introduced the notions of multi-sensitivity with respect to a vector, $\mathcal{N}$-sensitivity and strong multi-sensitivity for further understanding of multi-sensitivity and sensitivity \cite{JWLL2018}. They proved that if a dynamical system is multi-transitive with dense set of periodic points, then it is $\mathcal{N}$-sensitive.
 
Non-autonomous discrete dynamical systems  are  more complex dynamical systems and they  generalize the concept of autonomous discrete dynamical systems in a natural way.  They were introduced by Kolyada and Snoha in 1996 \cite{KS1996}, who gave a detailed motivation to the study of such systems and, in particular, their entropy. Non-autonomous discrete systems are related to the theory of difference equations and, in general, they provide a more adequate framework for the study of natural phenomena that appear in biology, physics, engineering, etc. (\cite{MR2855457, MR3378602}). In non-autonomous dynamical systems, the composition of two elements of the orbit of a point need not to be an element of the orbit.  As a consequence, the techniques used in this context are, in general, different from those used for autonomous systems and make this discipline of added interest.  Over the last decade, the study of chaos, sensitivity and stronger versions of sensitivity on  non-autonomous discrete systems has gained a lot of attention of researchers (\cite{KJA22, LZW2019, SD2019, SD2020, MR3584171, MR3846218}). Tian and Chen \cite{TC2006} introduced the notion of sensitivity in non-autonomous discrete dynamical systems. In 2015, Huang et al. \cite{HSZ2015} introduced the notion of cofinite sensitivity for non-autonomous systems and obtained several sufficient conditions for sensitivity and cofinite sensitivity in non-autonomous systems.  Recently, Salman and Das obtained a sufficient condition for non-autonomous $M$-systems under which multi-sensitivity, thick sensitivity and thickly syndetic sensitivity are equivalent \cite{SD2020c}.

The paper is organized as follows. In Section \ref{S-2}, we give prerequisites required for remaining sections of the paper. In Section \ref{S-3}, the notions of multi-sensitivity with respect to a vector, $\mathcal{N}$-sensitivity and strong multi-sensitivity are introduced  and studied. We first give an example showing that multi-sensitivity need not imply strong multi-sensitivity and hence need not imply  $\mathcal{N}$-sensitivity. Then we prove that for a general non-autonomous system on $[0,1]$, sensitivity  need not be not equivalent to $\mathcal{N}$-sensitivity  and the result is true for a periodic non-autonomous system on $[0,1]$.  Moreover, we prove that $\f\times g_{1,\infty}$ is strongly multi-sensitive (respectively, $\mathcal{N}$-sensitive) if and only if $\f$ or $g_{1,\infty}$ is strongly multi-sensitive (respectively, $\mathcal{N}$-sensitive). We also study these notions on a non-autonomous system, $(X, \W)$, induced by probability measure spaces and prove that if $(X, \W)$ is strongly multi-sensitive (respectively, $\mathcal{N}$-sensitive), then so is $(X, \f)$ and for the converse if  $(X, \f)$ is strongly multi-sensitive, then  $(X, \W)$ is $\mathcal{N}$-sensitive.  In Section \ref{S-4}, we first give examples showing that if $f_n$ converges to $f$ uniformly, then strong multi-sensitivity  or $\mathcal{N}$-sensitivity of $(X, \f)$ cannot be determined by that of the induced autonomous system $(X, f)$. In addition, we give a sufficient condition under which strong multi-sensitivity  (respectively, $\mathcal{N}$-sensitivity) of $(X, \f)$ coincides with that of $(X, f)$. Finally, we prove that for   a uniformly convergent non-autonomous system, if $(X, \f)$ is multi-transitive and has dense set of periodic points, then it is $\mathcal{N}$-sensitive.
\section{Preliminaries}\label{S-2}
 Let $\mathbb{N} =\{1,2,3, \ldots\}$ and $\mathbb{Z}^+=\{0,1,2, \ldots\}$. If $(X, d)$ is a metric space then  $B_d(x, \epsilon)$ denotes  the open ball with center $x\in X$ and radius $\epsilon$.
For each $n\in\N$, let $f_n: X\to X$ be a sequence of continuous functions on a compact metric space $X$. For convenience, denote $f_{1,\infty} := \{f_n\}_{n=1}^{\infty}$, and for all  $i$, $n$ and $k\in\N$, $$f^{n}_i:= f_{n+i-1}\circ f_{n+i-2}\circ\cdots\circ f_{i}, \ f_i^0 := id, \text{ and } f_{1,\infty}^{[k]} := \{f_{k(n-1)+1}^k\}_{n=1}^{\infty},$$    where $f_{1,\infty}^{[k]}$ is the {\it $k$-th iterate} of $\f$ (see \cite{MR3019971}). The $orbit$ of a point $x\in X$ under $\f$  is given by, $$\mathcal{O}_{f_{1,\infty}}(x)=\{x, f_1(x), f_2\circ f_1(x), f_3\circ f_2\circ f_1(x), \ldots \},$$ which is the solution of the following non-autonomous difference equation:
 $$\begin{cases}
x_{n+1}  & =  f_n(x_n), \\
x_0 & = x.\end{cases}$$
In particular, note that an autonomous dynamical system $(X, f)$ is a special case of the above system if $f_n = f$, for every $n\in\N$.  A point $x\in X$ is  {\it periodic} if there exists some $N\in \N$
such that $f_1^{nN}(x)=x$, for every $n\in\N$.

  We say that $(X, f_{1,\infty})$ is $(i)$ $periodic$, if there exists a $k\in\mathbb{N}$ such that $f_{j+kl} = f_j$, for any $l\in\mathbb{N}$ and for any  $1\leq j\leq k$; $(ii)$  \textit{feeble open}, if $int(f_n(U))\ne\varnothing$, for any nonempty open subset $U$ of $X$ and for any $n\in\N$. 

For any $m\geq 2$, if $(\f)^m = \{\underbrace{f_n\times\cdots\times f_n}_{m-\text{times}}\}_{n\in\N}$, where $(\f)^m =\f\times\cdots\times\f$ ($m$-times), then $(X^m, (\f)^m)$ is a non-autonomous dynamical system, where $X^m = X\times\cdots\times X$ ($m$-times). If $(X_1, d_1)$ and $(X_2, d_2)$ are two metric spaces, then we consider the product metric $\hat{d}$ on $X_1\times X_2$, where 
$$ \hat{d}\left((x_1, x_2), (x_1', x_2')\right) = \sqrt{d_1^2(x_1, x_1') +d_2^2(x_2, x_2')}, \text{ for all }  (x_1, x_2), (x_1', x_2')\in X_1\times X_2.$$
 A  sequence $\{f_n^k\}_{k\in\N}$ \textit{converges collectively} to $\{f^k\}_{k\in\N}$ with respect to the  supremum metric $D$ on $\mathcal{C}(X)$,  if for every $\epsilon>0$ there exists an $N_0\in\N$ such that $D(f_N^k, f^k)<\epsilon$, for all $N\geq N_0$  (see \cite{MR3846218}). 
 
 Let $\mathcal{M}(X)$ be the set of all Borel probability measures on $(X, \mathcal{B}(X))$, where $\mathcal{B}(X)$ is the $\sigma$-algebra of Borel subsets of $X$,   with the \textit{Prohorov metric $\mathcal{D}$} defined as  
 $$\mathcal{D}(\nu_1, \nu_2) = \inf\left\lbrace\epsilon: \nu_1(B)\leq \nu_2(N(B, \epsilon))+\epsilon, \text{ for every } B\in\mathcal{B}(X)\right\rbrace, $$
 where $N(B, \epsilon) = \bigcup_{b \in B} B_d(b, \epsilon)$.  Then $\mathcal{D}$ induces the \textit{weak*-topology} on $\mathcal{M}(X)$ \cite{MR2488795}. Let $\delta_x\in \mathcal{M}(X)$ be the \textit{Dirac point measure}, given by $\delta_x(A) = 0$, if $x\notin A$ and $\delta_x(A)=1$, if $x\in A$. For any $k\in\N$, we denote $$\mathcal{M}_k(X) =\left\lbrace\frac{1}{k}\left(\sum_{j=1}^k\delta_{x_j}\right) : x_j\in X \ \text{(not necessarily distinct)}\right\rbrace  \text{ and } \mathcal{M}_{\infty}(X) = \bigcup_{k\in\mathbb{N}}\mathcal{M}_k(X).$$ Then from \cite{MR0370540}, we  know that $\mathcal{M}_{\infty}(X)$ is dense in $\mathcal{M}(X)$ and $\mathcal{M}_k(X)$ is closed in $\mathcal{M}(X)$, for any $k\in\N$.  For $(X, f_{1,\infty})$, we will take the induced non-autonomous system $(\mathcal{M}(X), \widetilde{f}_{1,\infty})$, where $\widetilde{f}_{1,\infty} = \{\widetilde{f}_{i}\}_{i=1}^\infty$, $\widetilde{f}_{i}: \mathcal{M}(X)\to \mathcal{M}(X)$ is a continuous function, for each $i\in\N$ such that $\widetilde{f}_{1}^n(\nu)(B) = \nu(f_1^{-n}(B))$,  with $\nu\in \mathcal{M}(X)$, $B\in\mathcal{B}(X)$ and $f_1^{-n} = (f_1^n)^{-1}$ (see \cite{SD2019}).

For a non-autonomous system $(X, \f)$, $V\subseteq X$ and $\delta>0$, we denote:
$$N_{\f}(V, \delta) :=\left\lbrace n\in\N: \text{ there exist } u, v\in V \text{ satisfying } d(f_1^n(u), f_1^n(v))>\delta\right\rbrace.$$ 
\begin{defn}(\cite{HSZ2015, SD2019}) A non-autonomous system $(X, \f)$ is said to be
\begin{enumerate}[(1)]
\item \textit{sensitive}, if there exists a $\delta>0$ such that $N_{\f}(V, \delta)\ne\varnothing$, for any nonempty open subset $V$ of $X$;
\item \textit{cofinitely sensitive}, if there exists a $\delta>0$ such that $N_{\f}(V, \delta)$ is cofinite, for any nonempty open subset $V$ of $X$;
\item \textit{multi-sensitive}, if there exists a $\delta>0$ such that $\bigcap_{i=1}^mN_{\f}(V_i, \delta)\ne\varnothing$, for any collection of nonempty open subsets $V_1$, \ldots, $V_m$ of $X$.
\end{enumerate}
\end{defn}

\begin{defn}(\cite{SD2019, TC2006}) A non-autonomous system $(X, f_{1,\infty})$ is \textit{topologically transitive}, if there exists an $n\in\mathbb{N}$ such that $f_1^n(U)\cap V\ne\varnothing$, for any pair of nonempty open sets $U,$ $V\subseteq X$. A non-autonomous system $(X, \f)$ is called \textit{weakly mixing} if $\f\times\f$ is topologically transitive.
\end{defn}
A non-autonomous system $(X, \f)$ is said to be {\it Devaney chaotic} if  it is topologically transitive on $X$, it has dense set of periodic points on $X$ and it is sensitive on $X$.

\begin{defn}(\cite{SD2020}) A non-autonomous system $(X, \f)$ is  \textit{multi-transitive}, if $(X^m,$ $ \f^{[1]}\times\cdots\times \f^{[m]})$ is topologically transitive for any $m\in\N$, i.e., for any collection of nonempty open subsets $U_1$, $U_2$, \ldots, $U_m$; $V_1$, $V_2$, \ldots, $V_m$ of $X,$ there exists a $k\in\N$ such that $f_1^{ik}(U_i)\cap V_i\ne\varnothing$, for any $i\in\{1,2, \ldots, m\}$.
\end{defn}

\section{Multi-sensitivity with respect to a vector}\label{S-3}
\begin{defn} A non-autonomous system $(X, \f)$, for a vector $\mathbf{v}=(v_1, v_2, \ldots, v_r)\in\N^r$, is said to be 
\begin{enumerate}[(1)]
\item \textit{multi-sensitive with respect to vector $\mathbf{v}$,} if for any nonempty open subsets $U_1$, $U_2$, \ldots, $U_r$ of $X$, we have $\bigcap_{i=1}^r N_{f_{1,\infty}^{[v_i]}}(U_i, \delta)\ne\varnothing$, for some $\delta>0$;
\item \textit{$\mathcal{N}$-sensitive,} if $(X, \f)$ is multi-sensitive with respect to vector $\mathbf{v} = (1, 2, \ldots, n)$, for any $n\in\N$;
\item \textit{strongly multi-sensitive,} if $(X, \f)$ is multi-sensitive with respect to any vector in $\N^r$, for any $r\in\N$.
\end{enumerate}
\end{defn} 
It is easy to see that for non-autonomous dynamical systems: \\ Cofinite sensitivity $\implies$ strong multi-sensitivity $\implies$ $\mathcal{N}$-sensitivity $\implies$ sensitivity.

Note that strong multi-sensitivity implies multi-sensitivity for non-autonomous systems. The following example shows that for non-autonomous dynamical systems multi-sensitivity need not imply $\mathcal{N}$-sensitivity and hence need not imply strongly multi-sensitivity. 
\begin{exm} Let $\mathbb{S}^1$ be the unit circle with the arc length metric. Let $f_n: \mathbb{S}^1\to \mathbb{S}^1$ and $g_m:\mathbb{S}^1\to \mathbb{S}^1$ be defined as $f_1 = g_1 =id_{\mathbb{S}^1}$ and 
$$ f_n(e^{i\phi_1}) = e^{i\left(\frac{n}{n-1}\right)\phi_1}, \ \ g_m(e^{i\phi_2}) = e^{i\left(\frac{m-1}{m}\right)\phi_2}, \text{ for } n,m\geq 2.$$ Consider the non-autonomous system $(\mathbb{S}^1, h_{1,\infty})$, where $h_n$, for each $n\in\N$ is given by $h_{2n-1}(e^{i\phi}) = f_1^n(e^{i\phi})$ and $h_{2n}(e^{i\phi})= g_1^n(e^{i\phi})$, i.e., 
$$h_{1,\infty} = \{f_1, g_1, (f_2\circ f_1), (g_2\circ g_1), \ldots, (f_n\circ\cdots\circ f_1), (g_n\circ\cdots\circ g_1), \ldots\}. $$
Then using the fact that $f_1^m(e^{i\phi}) = e^{im\phi}$, for every $m\geq 1$ and $h_1^{2n-1}(e^{i\phi}) = f_1^n(e^{i\phi})$, for every $n\geq 1$, we get that $(\mathbb{S}^1, h_{1,\infty})$ is multi-sensitive. But as $h_1^{2n}(e^{i\phi}) = e^{i\phi}$, for every $n\geq 1$, we get that $(\mathbb{S}^1, h_{1,\infty})$ cannot be $\mathcal{N}$-sensitive or strongly multi-sensitive.
\end{exm} 
\begin{rmk}The above example also shows that  weakly mixing need not imply $\mathcal{N}$-sensitivity or strong multi-sensitivity for non-autonomous dynamical systems.
\end{rmk}

\begin{rmk} For autonomous dynamical systems Jiao et al. have proved that  $([0, 1],f)$ is sensitive if and only if it is multi-sensitive if and only if it is $\mathcal{N}$-sensitive if and only if it is strongly multi-sensitive \cite[Corollary 3.2.]{JWLL2018}. However, this is not always true for non-autonomous dynamical systems as seen by the following example. 
\end{rmk}
\begin{exm} Let $a$, $b$, $c$, $d$ be distinct rationals in $(0,1)$ and $A$ be the set of  such elements. For every $(a, b, c, d)\in A$, we consider a homeomorphism $f: [0, 1]\to [0, 1]$ such that $f(a) = c$ and $f(b) = d$. Let $\{f_n\}_{n\in\N}$ be an enumeration of such functions and consider the non-autonomous system $([0, 1], \f)$ as follows:
\[f_m(x) = \begin{cases}
(f_n)^{-1} & \text{if} \ m\in\{4n-3, 4n\}, \\
f_n & \text{if} \ m\in\{4n-2, 4n-1\}. 
\end{cases}  \] We claim that $([0, 1], \f)$ is sensitive. Let $U$ be any nonempty open subset of $[0, 1]$ and $x$, $y\in[0, 1]$ such that $|x-y|>3\delta$ and $W_1$, $W_2$ be the $\delta$-neighborhoods of $x$ and $y$, respectively. Since rationals in $(0, 1)$ are dense in $[0, 1]$, therefore there exist distinct rationals $\alpha$, $\beta\in U$, $\gamma\in W_1$ and $\eta\in W_2$. Moreover, there exists an $n\in\N$ such that $f_n(\alpha) = \gamma$ and $f_n(\beta) = \eta$. Now, the fact $|x-y|>3\delta$ and triangle inequality imply that $|f_n(\alpha)- f_n(\beta)|>\delta$. Then using $f_1^{4n-1} = f_n$, for each $n\in\N$, we get that $([0, 1], \f)$ is sensitive. But as $f_1^{4n} = id$, for each $n\in\N$, $([0, 1], \f)$ cannot be $\mathcal{N}$-sensitive or strongly multi-sensitive.
\end{exm}
Now, we show that for a periodic non-autonomous system the above equivalence is true. 
\begin{lem}\label{1} Let $(X, \f)$ be a $k$-periodic non-autonomous system. Then $(X, \f)$ is  strongly multi-sensitive (respectively, $\mathcal{N}$-sensitive) if and only if $(X, f_k\circ\cdots\circ f_1)$ is strongly multi-sensitive (respectively, $\mathcal{N}$-sensitive).
\end{lem}
\begin{proof}
Let $(X, \f)$ be strongly multi-sensitive and for every $r\in\N$, $\mathbf{v} = (v_1, v_2, \ldots, v_r)$ $\in \N^r$ be arbitrary. Let $U_1$, $U_2$, \ldots, $U_r$ be any collection of nonempty open subsets of $X$. Let $W_i = U_1$, $W_{i+k} = U_2$, \ldots, $W_{(r-1)k+i} = U_r$, for every $i\in\{1, 2, \ldots, k\}$. Now, using the fact that $(X, \f)$ is strongly multi-sensitive, we get existence of a $\delta>0$ such that $\bigcap_{i=1}^{rk} N_{f_{1,\infty}^{[v_i']}}(W_i, \delta)\ne\varnothing$, for the vector $\mathbf{v'}= (v_1', v_2', \ldots, v_{rk}')=(v_1, 2v_1, \ldots, kv_1, v_2, 2v_2, \ldots, kv_2, \ldots,$ $v_r, 2v_r, \ldots, kv_r)$. Therefore, for $n\in \bigcap_{i=1}^{rk} N_{f_{1,\infty}^{[v_i']}}(W_i$, $\delta)$ there exist  $w_i$, $w_i'\in W_i$ such that $d(f_1^{nv_i'}(w_i), f_1^{nv_i'}(w_i'))>\delta$, for every $i\in\{1, 2, \ldots, rk\}$. Consequently, there exist $u_j= w_{jk}\in U_j$ and $u_j'=w_{jk}'\in U_j$, for each $j\in\{1, 2, \ldots, r\}$ satisfying $d(f_1^{nkv_i}(u_i), f_1^{nkv_i}(u_i'))>\delta$, for every $i\in\{1, 2, \ldots, r\}$. Now, using the $k$-periodicity of $(X, \f)$, we get that $f_1^{nkv_i} = (f_1^k)^{nv_i} = (f_k\circ\cdots\circ f_1)^{nv_i}$, for every $i\in\{1, 2, \ldots, r\}$. Thus, $d(g^{nv_i}(u_i), g^{nv_i}(u_i'))>\delta$, where $g=  f_k\circ\cdots\circ f_1$. Therefore, $\bigcap_{i=1}^r N_{g^{[v_i]}}(U_i, \delta)\ne\varnothing$ and hence the autonomous system $(X, g)$ is strongly multi-sensitive.

Conversely, let $(X, g)$ be strongly multi-sensitive and for every $r\in\N$, $\mathbf{v} = (v_1, v_2, \ldots$, $v_r)\in \N^r$ be arbitrary. Let $U_1$, $U_2$, \ldots, $U_r$ be any collection of nonempty open subsets of $X$. Since $(X, g)$ is strongly multi-sensitive, we get that $\bigcap_{i=1}^r N_{g^{[v_i]}}(U_i, \delta)\ne\varnothing$, for some $\delta>0$. Then again using the fact that $(f_1^k)^s = f_1^{ks}$, for every $s\in\N$, it is easy to see that $\bigcap_{i=1}^{r} N_{f_{1,\infty}^{[v_i]}}(U_i, \delta)\ne\varnothing$. Hence, $(X, \f)$ is strongly multi-sensitive. Note that taking the vector $\mathbf{v} = (1, 2, \ldots, r)$ and proceeding as above we get that $(X, \f)$ is $\mathcal{N}$-sensitive  if and only if $(X, f_k\circ\cdots\circ f_1)$ is $\mathcal{N}$-sensitive.
\end{proof}
\begin{rmk} Note that by using Lemma \ref{1} we can construct many examples of a non-autonomous discrete dynamical system which is strongly multi-sensitive or $\mathcal{N}$-sensitive. Moreover, we have examples of a non-autonomous  system in which every member of $\f$ is strongly multi-sensitive (respectively, $\mathcal{N}$-sensitive) but the non-autonomous system itself is not even sensitive and vice versa. For instance,  one can see that in  \cite[Example 4.8]{SD2020c} every member of $\f$ is strongly multi-sensitive but the non-autonomous system  is not sensitive. Using Lemma \ref{1}, we can find examples of a non-autonomous system satisfying that the system is   strongly multi-sensitive but none of the maps $f_i$ is sensitive.
\end{rmk}
Using Lemma \ref{1}, \cite[Corollary 3.2]{JWLL2018} and proceeding as in \cite[Theorem 5.1]{SD2019}, we get the following.
\begin{thm}\label{2} For  a periodic non-autonomous system $([0, 1], \f)$ the following are equivalent:
\begin{enumerate}[(1)]
\item Strongly multi-sensitivity
\item $\mathcal{N}$-sensitivity
\item Multi-sensitivity
\item Sensitivity
\end{enumerate}
\end{thm}
\begin{cor}\label{MT-MS} If a periodic non-autonomous system $([0, 1], \f)$ is multi-transitive, then it is strongly multi-sensitive.
\end{cor}
\begin{proof}
Let $([0, 1], \f)$ be a $k$-periodic multi-transitive non-autonomous system and $g = f_k\circ\cdots\circ f_1$. Then by \cite[Proposition 3.1]{SD2020}, we get that  $([0, 1], g)$ is multi-transitive. Now, for autonomous systems it is well known that topological transitivity on intervals implies  sensitivity, therefore  $([0, 1], g)$ is sensitive and hence the non-autonomous system $([0,1], \f)$ is sensitive. Thus, using Theorem \ref{2}, we get that $([0, 1], \f)$ is strongly multi-sensitive.
\end{proof}
\begin{rmk} For non-autonomous dynamical systems on the closed unit interval, S\'anchez et al. \cite[Example 4.8]{MR3584171} have proved that topological transitivity need not imply Devaney chaos in contrast to the autonomous dynamical systems. Similar to Corollary \ref{MT-MS}, it can be verified that multi-transitivity on any interval implies Devaney chaos for periodic non-autonomous dynamical systems. 
\end{rmk}
As far as known to us following results  have not been studied for autonomous dynamical systems. 
\begin{thm} Let $(X, d_1)$ and $(Y, d_2)$ be two metric spaces and $(X, \f)$ and $(Y, g_{1,\infty})$ be two non-autonomous systems. Then the non-autonomous system $(X\times Y, \f\times g_{1,\infty})$ is strongly multi-sensitive (respectively, $\mathcal{N}$-sensitive) if and only if $(X, \f)$ or $(Y, g_{1,\infty})$ is strongly multi-sensitive (respectively, $\mathcal{N}$-sensitive). 
\end{thm}
\begin{proof}
Let $(X\times Y, \f\times g_{1,\infty})$ be strongly multi-sensitive. We need to show that $(X, \f)$ or $(Y, g_{1,\infty})$ is strongly multi-sensitive. Assume the contrary. Therefore, for every $\epsilon>0$ there exist $r_1$, $r_2\in\N$ such that for vectors $\mathbf{v} =(v_1, v_2, \ldots, v_{r_1})$,  $\mathbf{v'} =(v_1', v_2', \ldots, v_{r_1}')$ and nonempty open subsets $U_1$, $U_2$, \ldots, $U_{r_1}$ of $X$ and $U_1'$, $U_2'$, \ldots, $U_{r_1}'$ of $Y$, for every $m\in\N$, there exist $s_m\in\{1, 2, \ldots, r_1\}$ and $t_m\in\{1, 2, \ldots, r_2\}$ with \begin{equation}\label{3E1}
\text{diam}(f_1^{mv_{s_m}}(U_{s_m}))\leq \frac{\epsilon}{\sqrt{2}} \ \text{ and } \ \text{diam}(f_1^{mv_{t_m}'}(U_{t_m}'))\leq \frac{\epsilon}{\sqrt{2}}.
\end{equation}
Now, since each $U_i\times U_j'$, for $i\in\{1, 2, \ldots, r_1\}$ and $j\in\{1, 2, \ldots, r_2\}$ is a nonempty open subset of $X\times Y$, therefore by strong multi-sensitivity of $\f\times g_{1,\infty}$,  for the vectors $\mathbf{v}$ and $\mathbf{v'}$, we have $\bigcap_{i=1}^{r_1}\bigcap_{j=1}^{r_2} N_{f_{1,\infty}^{[v_i]}\times g_{1,\infty}^{[v_j']}}(U_i\times U_j', \delta)\ne\varnothing$, for some $\delta>0$. This implies that there exists an $n\in\N$ such that diam$(f_1^{nv_i}(U_i)\times g_1^{nv_j'}(U_j'))>\delta$, for every $i\in\{1, 2, \ldots, r_1\}$ and $j\in\{1, 2, \ldots, r_2\}$. Using \eqref{3E1} with $\epsilon=\delta$, we get the following contradiction.  
\[ \delta< \text{diam}\left(f_1^{nv_{s_n}}(U_{s_n})\times g_1^{nv_{t_n}'}(U_{t_n}')\right)\leq \sqrt{\left(\frac{\delta}{\sqrt{2}}\right)^2+\left(\frac{\delta}{\sqrt{2}}\right)^2}= \delta.\]
Converse follows by using the fact that for any $r\in\N$ and vectors $\mathbf{v} =(v_1, v_2, \ldots, v_{r})$,  $\mathbf{v'} =(v_1', v_2', \ldots, v_{r}')$, \[\bigcap_{i=1}^{r} N_{f_{1,\infty}^{[v_i]}\times g_{1,\infty}^{[v_i']}}\left(U_i\times U_i', \delta\right)\supseteq \left(\bigcap_{i=1}^{r} N_{f_{1,\infty}^{[v_i]}}(U_i, \delta)\right)\bigcup\left( \bigcap_{i=1}^{r} N_{ g_{1,\infty}^{[v_i']}}(U_i', \delta)\right),\] for some $\delta>0$ and for every collection of nonempty open subsets $U_1$, $U_2$, \ldots, $U_{r}$ of $X$ and $U_1'$, $U_2'$, \ldots, $U_{r}'$ of $Y$. The proof for $\mathcal{N}$-sensitivity  is similar.
\end{proof}
\begin{thm} Let $(X, \f)$ be a non-autonomous system. If $(\M, \W)$ is strongly multi-sensitive (respectively, $\mathcal{N}$-sensitive), then $(X, \f)$ is strongly multi-sensitive (respectively, $\mathcal{N}$-sensitive).
\end{thm}
\begin{proof}
Let $(\M, \W)$ be strongly multi-sensitive and assume that $(X, \f)$ is not strongly multi-sensitive. Then  for every $\delta>0$, there exists an  $r\in\N$ such that for every vector $\mathbf{v}=(v_1, \ldots, v_r)$ and nonempty open subsets $U_1$, $U_2$, \ldots, $U_r$ of $X$, for every $n\in\N$ there exist $s_n\in\{1,2, \ldots, r\}$ such that diam($f_1^{nv_{s_n}}(U_{s_n})\leq \delta/2$. For $x_{s_n}\in U_{s_n}$, we can take $B_d(x_{s_n}, \xi)\subseteq U_{s_n}$, for $0<\xi<\delta/2$ which implies that  $$\text{diam}\left(f_1^{nv_{s_n}}(B_d(x_{s_n}, \xi)\right)\leq \delta/2, \text{ for every } n\in\N.$$ By the definition of Prohorov metric, for any $\mu\in B_{\mathcal{D}}(\delta_{x_{s_n}}, \xi)$ and for any $B\in\mathcal{B}(X)$, there exists an  $0<\epsilon<\xi$ such that $\delta_{x_{s_n}}(B)\leq \mu(N(B, \epsilon))+\epsilon$ implying that $\mu(B_d(x_{s_n}, \xi))\geq (1-\epsilon)$, for every $n\in\N$. For any $B\in\mathcal{B}(X)$ and for any $n\in\N$, we have the following cases:
\begin{enumerate}
\item[Case 1.] If $f_1^{nv_{s_n}}(x_{s_n})\notin B$, then $ 0= \delta_{x_{{s_n}}}(f_1^{-nv_{s_n}}(B))\leq \mu(f_1^{-nv_{s_n}}(N(B, \delta)))+\delta$.
\item[Case 2.] If $f_1^{nv_{s_n}}(x_{s_n})\in B$, then it can be verified that $1 = \delta_{x_{{s_n}}}(f_1^{-nv_{s_n}}(B))\leq (1-\epsilon +\delta)\leq \mu(f_1^{-nv_{s_n}}(N(B, \delta)))+\delta$.
\end{enumerate}
Therefore, from the above cases, we get that $\mathcal{D}(\widetilde{f}_1^{nv_{s_n}}(\delta_{x_{{s_n}}}), \widetilde{f}_1^{nv_{s_n}}(\mu))\leq \delta$, for every $n\in\N$ and for every $\delta>0$, which is a contradiction to the fact that $(\M, \W)$ is strongly multi-sensitive. Hence, $(X, \f)$ is strongly multi-sensitive. Similarly, if $(\M, \W)$ is $\mathcal{N}$-sensitive, then it can be verified that $(X, \f)$ is $\mathcal{N}$-sensitive.
\end{proof}
For the converse, we have the following.
\begin{lem}\label{L3}(\cite[Theorem 4.3]{SD2019b}) A non-autonomous system $(X, \f)$ is multi-sensitive if and only if $(\M, \W)$ is multi-sensitive.
\end{lem}
\begin{thm}
If $(X, \f)$ is strongly multi-sensitive, then $(\M, \W)$ is $\mathcal{N}$-sensitive.
\end{thm}
\begin{proof}
Since $(X, \f)$ is strongly multi-sensitive, therefore it is multi-sensitive with respect to the vector $\mathbf{v} =(\underbrace{1,\ldots, 1}_{k_1}, \underbrace{2,\ldots, 2}_{k_2}, \ldots, \underbrace{k, \ldots, k}_{k_s})$, for any $k$, $k_i$ and $s\in\N$, $1\leq i\leq s$. Thus, proceeding as in the proof of the converse of Lemma \ref{L3}, yields that $(\M, \W)$ is $\mathcal{N}$-sensitive.
\end{proof}
By Theorem \ref{2} and Lemma \ref{L3}, we get the following.
\begin{thm} Let $([0, 1], \f)$ be a periodic non-autonomous system. Then $([0, 1], \f)$ is strongly multi-sensitive (respectively, $\mathcal{N}$-sensitive) if and only if $(\mathcal{M}([0, 1]), \W)$ is strongly multi-sensitive (respectively, $\mathcal{N}$-sensitive).
\end{thm}

\section{Multi-transitivity and $\mathcal{N}$-sensitivity in uniformly convergent non-autonomous systems}\label{S-4}
For a uniformly convergent non-autonomous system, we are interested in comparing the dynamical behavior of the non-autonomous system and the autonomous system induced by the uniform limit. So, if $\f$ converges to $f$ uniformly, then it is natural to ask that if $(X, f)$ is strongly multi-sensitive (respectively, $\mathcal{N}$-sensitive), then what about the non-autonomous system $(X, \f)$ and vice versa. The following examples show that a uniformly convergent non-autonomous system which is  strongly multi-sensitive  can converge to a limit function which is not sensitive and vice versa.
\begin{exm}\label{Ex2} Let $([0,2], \f)$ be the uniformly convergent  non-autonomous system, where 
\[ f_1(x) = f(x) =  \begin{cases}
2x & \text{if} \ x \in \left[0, \frac{1}{2}\right] \\
2(1-x) & \text{if} \ x \in \left[\frac{1}{2}, 1\right] \\
4(x-1)(2-x) & \text{if} \ x\in[1, 2], 
\end{cases}\]
\[ f_n(x) = g(x) = \begin{cases}
2x & \text{if} \ x \in \left[0, \frac{1}{2}\right] \\
2-2x & \text{if} \ x \in \left[\frac{1}{2}, 1\right] \\ 
x-1 & \text{if} \ x \in \left[1, 2\right],
\end{cases} \  (n\geq 2). \]
Then $\f$ converges uniformly to $g$ which is not sensitive as it is isometry on $[1, 2]$. Moreover, it can be verified that for any $x\in[0, 2]$, $f_1^n(x) = f^n(x)$, for any $n\in\N$ and $f$ is strongly multi-sensitive. Consequently, we get that the non-autonomous system $([0,2], \f)$ is also strongly multi-sensitive.
\end{exm}
\begin{exm}\label{Ex3} Let $([0, 1], \f)$ be the non-autonomous discrete system, where the members of $\f$ are given as follows:
\[ f_1(x) = \begin{cases}
x & \text{if} \ x \in \left[0, \frac{1}{2}\right] \\
\frac{1}{2} & \text{if} \ x \in \left[\frac{1}{2}, 1\right],
\end{cases} \ \ f_2(x) = \begin{cases}
\frac{1}{3} & \text{if} \ x \in \left[0, \frac{1}{2}\right] \\
\frac{2x}{3} & \text{if} \ x \in \left[\frac{1}{2}, 1\right], 
\end{cases}  \]
\[ f_n(x) = f(x) = \begin{cases}
1-2x & \text{if} \ x \in \left[0, \frac{1}{2}\right] \\
4x-2 & \text{if} \ x \in \left[\frac{1}{2}, \frac{3}{4}\right] \\ 
-2x+\frac{5}{2} & \text{if} \ x \in \left[\frac{3}{4}, 1\right]
\end{cases}  \ (n\geq 3). \]
Then $\f$ converges uniformly to the continuous map $f$ and $f$ is strongly multi-sensitive and hence $\mathcal{N}$-sensitive. But since  for any $0 \leq x \leq 1/2$, we have $f_1^{n}(x) = 1/3$, for every $n\in\N$, therefore the non-autonomous system $([0, 1], \f)$ cannot be sensitive and hence cannot be $\mathcal{N}$-sensitive or strongly multi-sensitive.
\end{exm}

Now, we give a sufficient condition under which strong multi-sensitivity (respectively, $\mathcal{N}$-sensitivity) of $(X, \f)$ implies and is implied by  that of $(X, f)$.

\begin{thm}\label{T2} Let $\f$ be a surjective  uniformly convergent sequence converging to $f$ such that $\{f_n^k\}_{k\in\N}$ converges collectively to $\{f^k\}_{k\in\N}$ and $\f$ is feebly open. Then the non-autonomous system $(X, \f)$ is  strongly multi-sensitive (respectively, $\mathcal{N}$-sensitive) if and only if $(X, f)$ is strongly multi-sensitive (respectively, $\mathcal{N}$-sensitive).
\end{thm}
\begin{proof}
Let $(X, \f)$ be strongly multi-sensitive. We need to show that $(X, f)$ is strongly multi-sensitive. Let $r\in\N$ be arbitrary, $U_1$, $U_2$, \ldots, $U_r$ be any collection of nonempty open subsets of $X$ and $\mathbf{v} =(v_1, v_2, \ldots, v_r)$ be any vector in $\N^r$. Let $\delta>0$ be the constant of strong multi-sensitivity for $\f$ and choose an $m\in\mathbb{N}$ such that $(1/m)<(\delta/3)$. Now, the fact that $\{f_n^k\}_{k\in\mathbb{N}}$ converges collectively to $\{f^k\}_{k\in\mathbb{N}}$, implies that there exists a $t\in\N$ such that $D(f_s^k, f^k)<1/m$, for any $s\geq t$ and for any $k\in\N$. Since $f_1^{tv_{i}-1}(x)\in X$, for every $i\in\{1, 2, \ldots, r\}$, therefore we get that $d(f_{v_it}^{v_ik}(f_1^{v_it-1}(x)), f^{v_ik}(f_1^{v_it-1}(x)))<1/m$ and hence for each $i\in\{1, 2, \ldots, r\}$, we have \begin{equation}\label{4e4} d(f_1^{(t+k)v_i}(x), f^{v_ik}(f_1^{v_it-1}(x)))<1/m, \ \text{for each } x\in X \text{ and for each} \ k\in\N. \end{equation}
 \textit{Claim 1.} $\bigcap_{i=1}^{r} N_{f_{1,\infty}^{[v_i]}}(U_i, \delta)$ is an infinite set.
 
 \textit{Proof of Claim 1.} Assume the contrary and let $l=\max\{n\in\N: \text{ there exist } u_i, u_i'\in U_i \text{ such that } d(f_1^{nv_i}(u_i), f_1^{nv_i}(u_i'))>\delta$, for each $i\in\{1, 2, \ldots, r\}\}$. Let $\epsilon>0$  be such that $\epsilon\leq \delta/2$ and $u_i\in U_i$, for each $i\in\{1, 2, \ldots, r\}$. By continuity of $f_1^{v_ij}$, for $j\in\{1, 2, \ldots, l\}$, there exists an $\eta>0$ such that \[ d(u_i, u_i')<\eta \ \implies d(f_1^{v_ij}(u_i), f_1^{v_ij}(u_i'))<\epsilon, \text{ for every } u_i'\in X.\]
Note that $\bigcap_{i=1}^{r} N_{f_{1,\infty}^{[v_i]}}((U_i\cap B(u_i, \eta)), \delta)\ne\varnothing$. Now, since $f_1^{v_ij}(U_i\cap B(u_i, \eta))\subseteq B(f_1^{v_ij}(u_i)$, $\epsilon)$, for every $i\in\{1, 2, \ldots, r\}$, $j\in\{1, 2, \ldots, l\}$, therefore there exists a $p>l$ and $u_i$, $w_i\in U_i\cap B(u_i, \eta)$ such that $d(f_1^{pv_i}(u_i), f_1^{pv_i}(w_i))>\delta$, for every $i\in\{1, 2, \ldots, r\}$, which is a contradiction to the maximality of $l$.

 Then by Claim 1, we get that $\bigcap_{i=1}^tN_{\f^{v_i}}((f_1^{v_it-1})^{-1}(U_i), \delta)$ is infinite and hence  there exist $x_i'$, $y_i'\in (f_1^{v_it-1})^{-1}(U_i)$ and a $p\in\N$ satisfying $d(f_1^{(p+t)v_i-1}(x_i')$, $f_1^{(p+t)v_i-1}(y_i'))>\delta$, for each $i\in\{1, 2, \ldots, r\}$. Thus,  there exist $x_i$, $y_i\in U_i$ such that $x_i = f_1^{v_it-1}(x_i')$ and $y_i = f_1^{v_it-1}(y_i')$, for each $i\in\{1, 2, \ldots, r\}$ implying that $d(f_{v_it}^{v_ip}(x_i), f_{v_it}^{v_ip}(y_i))>\delta$. Consequently, \eqref{4e4} and triangle inequality, imply that $$d(f^{v_ip}(x_i), f^{v_ip}(y_i))>(\delta -(2/m))>\delta/3, \text{ for each } i\in\{1, 2, \ldots, r\}$$ and hence $\bigcap_{i=1}^rN_{f^{v_i}}(U_i, \delta/3)\ne\varnothing$. Therefore, $(X, f)$ is strongly multi-sensitive. 

Conversely, let $(X, f)$ be strongly multi-sensitive and let $r\in\N$ be arbitrary, $u_1$, $u_2$, \ldots, $u_r\in X$, $U_i =B(u_i, \epsilon)$ and $\mathbf{v} =(v_1, v_2, \ldots, v_r)$ be any vector in $\N^r$. Let $\delta>0$ be the constant of strong multi-sensitivity for $f$ as chosen above. Since  $\f$ is feebly open, therefore each of  $U_i' = \text{int}(f_1^{v_it-1}(U_i))$ is a nonempty open subset of $X$, where $t$ is from \eqref{4e4}. By strong multi-sensitivity of $(X, f)$, there exist an $n\in\N$ and $x_i'$, $y_i'\in U_i'$ such that $d(f^{nv_i}(x_i'), f^{nv_i}(y_i'))>\delta$, for each $i\in\{1, 2, \ldots, r\}$. Now, as $x_i'$, $y_i'\in U_i'$, there exist $x_i$, $y_i\in U_i$ such that $x_i'= f_1^{v_it-1}(x_i)$ and $y_i' = f_1^{v_it-1}(y_i)$, for each $i\in\{1, 2, \ldots, r\}$. Therefore,  $d(f^{nv_i}(f_1^{v_it-1}(x_i)), f^{nv_i}(f_1^{v_it-1}(y_i)))>\delta$, for each $i\in\{1, 2, \ldots, r\}$ and hence by using \eqref{4e4} and triangle inequality, we get that $d(f_1^{(t+n)v_i}(x_i), f_1^{(t+n)v_i}(y_i))>\delta/3$, for each $i\in\{1, 2, \ldots, r\}$ and hence $\bigcap_{i=1}^rN_{\f^{[v_i]}}(U_i, \delta/3)\ne\varnothing$. Thus, $(X, \f)$ is strongly multi-sensitive. Note that taking the vector $\mathbf{v} = (1, 2, \ldots, r)$, for any $r\in\N$ and proceeding as above we get that   the non-autonomous system $(X, \f)$ is  $\mathcal{N}$-sensitive if and only if $(X, f)$ is $\mathcal{N}$-sensitive.
\end{proof}
\begin{rmk}
\begin{enumerate}[(1)]
 \item Note that $\f$ in Example \ref{Ex2} is such that for any $n\geq 2$, $\{f^k_n\}_{k\in\N}$  converges collectively to $\{f^k\}_{k\in\N}$. But the family $\f$ is not surjective.
 \item Note that $\f$ in Example \ref{Ex3} is such that for any $n\geq 3$, $\{f^k_n\}_{k\in\N}$  converges collectively to $\{f^k\}_{k\in\N}$. But the family $\f$ is not feebly open.
 \item Note that using similar arguments Theorem \ref{T2} is also true for multi-sensitivity.
 \end{enumerate}
\end{rmk}

As an application of Theorem \ref{T2}, using \cite[Theorem 2.6]{MR3846218}, we get another sufficient condition under which sensitivity of $\f$ is equivalent to strong multi-sensitivity of $\f$ on the unit interval $[0, 1]$.
\begin{thm} Let $([0, 1], \f)$ be a uniformly convergent non-autonomous system with $f$ as the uniform limit such that $\{f_n^k\}_{k\in\N}$ converges collectively to $\{f^k\}_{k\in\N}$ and $\f$ is feebly open. Then the following are equivalent:
\begin{enumerate}[(1)]
\item $([0, 1], \f)$ is strongly multi-sensitive
\item $([0, 1], \f)$ is $\mathcal{N}$-sensitive
\item $([0, 1], \f)$ is multi-sensitive
\item $([0, 1], \f)$ is sensitive.
\end{enumerate}
\end{thm}
\begin{lem}\label{L1}(\cite[Lemma 2.5]{MR3779662}) Let $(X, d)$ be a metric space without isolated points and $\f$ converges uniformly to $f$. Then, 
\begin{enumerate}[(a)]
\item If $p$ is a periodic point for $(X, \f)$, it is a periodic point for $(X, f)$.
\item If there exists an infinite set of periodic points of $f$, then there exists a $\xi>0$ such that for any $x\in X$ there is a periodic point $p$ of $f$ such that $d(x, f^n(p))\geq \xi$, for all $n\geq 0$.
\end{enumerate}
\end{lem}
\begin{lem}(\cite[Lemma 2.1]{MR3019971})\label{L4} If $f_n$ converges to $f$ uniformly on $X$, then $f_n^k$ converges uniformly to $f^k$, for every $k\in\N$.
\end{lem}

\begin{lem}(\cite[Lemma 3.1]{SD2020})\label{L2} If ($X, \f)$ is multi-transitive, then  the set $\{l\in\N : f_1^{jl}(U_j)\cap V_j\ne\varnothing$, for each $1\leq j\leq m\}$, for any collection of nonempty open sets $U_1$, $U_2$, \ldots, $U_m$; $V_1$, $V_2$, \ldots, $V_m$, is infinite.
\end{lem}
\begin{thm}\label{MS} Let $(X, d)$ be a metric space without isolated points and $\f$ converge uniformly to $f$. If $(X, \f)$ is multi-transitive and the set of periodic points of $\f$ is dense in $X$, then $(X, \f)$ is $\mathcal{N}$-sensitive. 
\end{thm}
\begin{proof}
Let $u_1$, $u_2$, \ldots, $u_r\in X$ and $\epsilon>0$ be given. Since the set of periodic points of $\f$ is dense and hence infinite, therefore Lemma \ref{L1} implies that corresponding to each $i\in\{1, 2, \ldots, r\}$ there exists a periodic point $p_i\in X$ such that \begin{equation}\label{2E1}
d(u_i, f^m(p_i))\geq \xi, \text{ for every } m\geq 0 \text{ and for each } i\in\{1, 2, \ldots, r\} .
\end{equation} Moreover, using density of periodic points of $\f$ in $X$, we get a  collection of periodic points $q_1$, $q_2$, \ldots, $q_r\in X$ satisfying $d(u_i, q_i)<\eta$, where $\eta =\min\{\epsilon, \xi/4\}$, for each $i\in\{1, 2, \ldots, r\}$. If $n_i$ is the period of $q_i$, then $f_1^{mn_{i}}(q_i) = q_{i}$, for every $m\in\N$ and for each $i\in\{1, 2, \ldots, r\}$. Let $n=lcm\{n_1,n_2,\ldots, n_r\}$. By uniform  continuity of $f^{il}$, for $l\in\{0, 1, 2, \ldots, n\}$ and $i\in\{1, 2, \ldots, r\}$, there exists an open  neighborhood $V_i$ of $p_i$ corresponding to each $i$, such that  \begin{equation}\label{2E2}
d(f^{il}(p_i), f^{il}(x_i))<{\xi}/{8}, \text{ for every } x_i\in V_i, \ i\in\{1, 2, \ldots, r\}.
\end{equation}  Let $U_i=B_d(u_i, \epsilon)$ and $V_i$ be the neighborhood of $p_i$, for each $i\in\{1, 2, \ldots, r\}$.  Now, by Lemma \ref{L4}, $f_{i((N-1)n+l)+1}^{(n-l)i}$ converges  uniformly to $f^{(n-l)i}$ as $N\to\infty$, for any $0\leq l<n$, $i\in\{1, 2, \ldots, r\}$, therefore there exists an $N_0\in \N$ such that for all $N\geq N_0$, we have \begin{equation}\label{2E3} d\left(f_{i((N-1)n+l)+1}^{(n-l)i}(v_i), f^{(n-l)i}(v_i)\right)<{\xi}/{8},\text{ for every } v_i\in V_i, \ i\in\{1, 2, \ldots, r\}.\end{equation}
By multi-transitivity of $(X, \f)$ there exists a $k\in\N$ such that $f_1^{ik}(U_i)\cap V_i\ne\varnothing$, for each $1\leq i\leq r$. Using Lemma \ref{L2}, we get that the set $\{k\in\N : f_1^{ik}(U_i)\cap V_i\ne\varnothing$, for each $1\leq i\leq r\}$ is an infinite set, therefore we can choose $s= (N-1)n+l$, $0\leq l<n$ such that $N\geq N_0$ and $f_1^{is}(U_i)\cap V_i\ne\varnothing$,  for every $i\in\{1, 2, \ldots, r\}$. Consequently, there exists a $y_i\in U_i$ such that $f_1^{is}(y_i)\in V_i$, for each $i\in\{1, 2, \ldots, r\}$. Hence, using \eqref{2E3}, we get that $$d\left(f_{i((N-1)n+l)+1}^{(n-l)i}\left(f_1^{i((N-1)n+l)}\right)(y_i), f^{(n-l)i}\left(f_1^{i((N-1)n+l)}\right)(y_i)\right)<{\xi}/{8},$$ which implies that $d(f_1^{iNn}(y_i), f^{(n-l)i}(f_1^{is}(y_i))<{\xi}/{8}$ for each $i\in\{1, 2, \ldots, r\}$.  This together with \eqref{2E2} and triangle inequality, yields that $d(f^{(n-l)i}(p_i), f_1^{inN}(y_i))<\xi/4$, for each $i\in\{1, 2, \ldots, r\}$. Finally using \eqref{2E1}, for each $i\in\{1, 2, \ldots, r\}$, we get that \begin{align*}   d(u_i, f^{(n-l)i}(p_i)) & \leq  d(f^{(n-l)i}(p_i), f_1^{inN}(y_i)) + d(q_i, f_1^{inN}(y_i))+ d(q_i, u_i) \\
\implies d(f_1^{inN}(q_i), f_1^{inN}(y_i)) & \geq d(u_i, f^{(n-l)i}(p_i))- d(f^{(n-l)i}(p_i), f_1^{inN}(y_i))-d(q_i, u_i) \\ & >\xi-{\xi}/{4}-{\xi}/{4} = {\xi}/{2}.
\end{align*} 
So either $d(f_1^{inN}(u_i), f_1^{inN}(q_i)) >\xi/4$ or $d(f_1^{inN}(u_i), f_1^{inN}(y_i)) >\xi/4$, for each $i\in\{1, 2, \ldots, r\}$. Note that $u_i$, $y_i\in U_i$, for each $i\in\{1, 2, \ldots, r\}$. Therefore, for $\delta=\xi/4$, we get that $\bigcap_{i=1}^r N_{\f^{[i]}}(U_i, \delta)\ne\varnothing$ and hence $(X, \f)$ is $\mathcal{N}$-sensitive.
\end{proof} 
\begin{rmk} Theorem \ref{MS} generalizes the main result of \cite{JWLL2018} to a uniformly convergent non-autonomous discrete dynamical system.
\end{rmk}
{

\end{document}